\documentclass[11pt,reqno,oneside]{amsart}
\usepackage{amssymb}
\usepackage{graphics}

\usepackage[usenames]{color}

\usepackage{hyperref}

\def\Box{\setlength{\unitlength}{0.01cm}
         \begin{picture}(15,15)(-5,-5)
           \framebox(15,15){}
         \end{picture} }
\newtheorem{theorem}{Theorem}[section]
\newtheorem{lemma}{Lemma}[section]

\newtheorem{rem}{Remark}[section]
\newtheorem{defi}{Definition}[section]
\newtheorem{cex}{Counterexample}[section]

\def\proof{\noindent{\bf Proof.}\,\,}

\def\R{{\rm I}\!{\rm R}}

\newcommand{\bit}{\begin{itemize}}
\newcommand{\eit}{\end{itemize}}
\newcommand{\bq}{\begin{equation}}
\newcommand{\eq}{\end{equation}}

\newcommand{\ds}{\displaystyle}

\newcommand{\hphi}{\widehat{\phi}}
\newcommand{\hx}{\hat{x}}
\newcommand{\hy}{\hat{y}}
\newcommand{\hm}{\widehat{M}}
\newcommand{\at}{\tilde{a}}
\newcommand{\bt}{\tilde{b}}
\newcommand{\hk}{\widehat K}

\newcommand{\oK}{\overline K}
\newcommand{\oQ}{\overline Q}
\newcommand{\hQ}{\widehat Q}
\newcommand{\hu}{\widehat u}
\newcommand{\ou}{\overline u}

\newcommand{\comp}{\underset{C}{\sim}}
\newcommand{\q}{\mathcal Q}

\makeatletter
\g@addto@macro{\endabstract}{\@setabstract}
\newcommand{\authorfootnotes}{\renewcommand\thefootnote{\@fnsymbol\c@footnote}}%
\makeatother

\keywords{Quadrilateral Elements, Anisotropic Finite Elements, Lagrange Interpolation, Minimum Angle Condition, Maximum Angle Condition.} \subjclass{65N15,65N30}
\author[G. Acosta]{Gabriel Acosta}
\address{Departamento de Matem\'atica\\ Facultad de
Ciencias Exactas y Naturales\\Universidad de Buenos Aires\\ IMAS-CONICET, (1428)
Buenos Aires\\ Argentina.}
\email{gacosta@dm.uba.ar}
\author[G. Monz\'on]{Gabriel Monz\'on}
\address{Instituto de Ciencias\\ Universidad Nacional de General Sarmiento\\
J. M. Guti\'errez 1150 \\ 
(1613) Los Polvorines\\
Buenos Aires\\ Argentina.}
\email{gmonzon@ungs.edu.ar}
\thanks{The first author is a member of CONICET. This work is partially supported by grants 
PIP 11220130100184CO 
and UBACYT 20020130100205BA} 

\title[The minimal angle condition for quadrilateral  finite elements]{ The minimal angle condition for quadrilateral  finite elements
of arbitrary degree}

\begin{document}
\maketitle
\begin{abstract}  We study $W^{1,p}$ Lagrange interpolation error estimates for general \emph{quadrilateral} $\q_{k}$ finite elements with $k\ge 2$. For
the most standard case of $p=2$ it turns out that the constant $C$ involved in the error estimate can be bounded in terms of the \emph{minimal interior angle} of the quadrilateral. Moreover, the same holds for any $p$ in the range $1\le p<3$. On the other hand, for $3\le p$ we show that $C$ also depends  on the 
\emph{maximal interior angle}. We provide some counterexamples showing that our results are sharp. 

\end{abstract}

\section{Introduction}
\label{section1}
\setcounter{equation}{0}

This paper deals with error estimates in the $W^{1,p}$ norm for the   $\q_k$ Lagrange interpolation on a general convex quadrilateral $K\subset \R^2$. Denoting the interpolant with $Q_k$
the standard error estimate is usually found in the form
\bq
\label{eq:gralestimate}
\left\| u-Q_k u \right\|_{0,p,K} + h \left| u-Q_ku \right|_{1,p,K} \leq Ch^{k+1} |u|_{k+1,p,K},
\eq
being $h$ the diameter of $K$. Inequality \eqref{eq:gralestimate} involves the  $L^p$ error  estimate 
\bq
\label{errorinLp}
\left\| u-Q_k u \right\|_{0,p,K} \leq Ch^{k+1} |u|_{k+1,p,K},
\eq
and the seminorm estimate
\bq
\label{errorinwp}
\left| u-Q_ku \right|_{1,p,K} \leq Ch^k |u|_{k+1,p,K}.
\eq
A central matter of \eqref{eq:gralestimate} concerns the dependence of $C$ 
on basic geometric quantities of the underlying element $K$. It is known that the constant $C$ in \eqref{errorinLp} remains uniformly bounded for arbitrary convex quadrilaterals (see Theorem \ref{teoerrornormap}).  However  this statement is false for  the constant $C$ in \eqref{errorinwp} (see for instance the counterexamples in the last section). The primary goal of this paper is to study the dependance of $C$ in
\eqref{errorinwp} on the \emph{interior angles} of $K$.  Although the role of the interior angles have been related to $C$ in many previous works, none of them, to the best of the authors knowledge,  have given a result as plain as the one offered in this paper. For instance, bounding
the minimal and the maximal inner angle is considered a central matter in  mesh generation algorithms since the early work by Ciarlet and Raviart \cite{CR}, however no proof of sufficiency has been given as far (at least for an arbitrary  degree of interpolation). 

In order to present our results let us first introduce the following classical definition that we write for both  triangles and quadrilaterals for further convenience.
\begin{defi}
Let $K$ (resp. $T$) be a convex quadrilateral (resp. a triangle). We say that $K$ (resp. $T$) satisfies the {\slshape{minimum angle condition}} with constant $\psi_m \in \R$, or shortly $mac(\psi_m)$, if for any internal angle $\theta$ of $K$ (resp. $T$) $0< \psi_m \leq \theta$.
\end{defi}
Our first result says that the constant in \eqref{errorinwp}, for a fixed
degree $k$ and a fixed value of $p$ with $1\le p<3$, can be written as
$C=C(\psi_m)$. As a consequence, the same can be stated about the constant in   \eqref{eq:gralestimate}. This seems to be the most general result available for quadrilaterals
in the case $k\ge 2$. In spite of the fact that much weaker geometrical conditions are known to be sufficient for the case $k=1$,  we show, by means of a counterexample, that they fail for a higher degree interpolation. This counterexample also warns that removing the minimal angle condition may indeed lead to a blowing-up constant in \eqref{errorinwp}. 

The $mac$ is  the most standard condition considered in textbooks for \emph{triangular} finite elements. Actually, in that case, it is equivalent to the so called \emph{regularity condition}, i.e. equivalent to the existence of a constant $\sigma$ such that
\bq
\label{regcond}
h/\rho \leq \sigma,
\eq
where $\rho$ denotes the diameter of the maximum circle contained in $T$. On the other hand, the term {\it{anisotropic}} or {\it{narrow}} is usually applied to  elements that do not satisfy \eqref{regcond}. Even when  \emph{triangles} can become narrow only if the minimal angle is approaching zero a very different situation occurs on quadrilaterals. Indeed, in that case the $mac$ condition is less restrictive than \eqref{regcond} since arbitrarily narrow elements are allowed with a positive uniform bound for the minimal angle (for example, anisotropic rectangles always verify $mac(\pi/2)$). Anisotropic  elements are important for instance in problems involving singular layers and the first works dealing with them arise during the seventies showing that  \eqref{regcond} can be replaced (for triangles) by the 
weaker following condition
\begin{defi}
Let $K$ (resp. $T$) be a convex quadrilateral (resp. triangle), we say that $K$ (resp. $T$) satisfies the {\slshape{maximum angle condition}} with constant $\psi_M \in \R$, or shortly $MAC(\psi_M)$, if for any internal angle $\theta$ of $K$ (resp. $T$) $\theta \leq \psi_M < \pi$.
\end{defi}
Indeed, in \cite{BA,J1} it is proved that the $MAC$ is  sufficient  to have optimal order error estimates for Lagrange interpolation on \emph{triangles}. 
In the case of quadrilateral elements, \eqref{regcond} it is also a sufficient
condition as it was  shown by Jamet \cite{J2} for $k=1$ and $p=2$. This condition is less restrictive than that propoposed in \cite{CR} where the authors require the existence of two positive constants $\mu_1, \mu_2$ such that
\bq
\label{shortside}
h/s \leq \mu_1
\eq
where $s$ is the length of the shortest side of $K$, and
\bq
\label{innerangles}
|\cos(\theta)| \leq \mu_2 <1
\eq
for each inner angle $\theta$ of $K$. 
Observe that under the regularity condition (\ref{regcond}) the quadrilateral can degenerate into a triangle (for instance if the shortest side tends to zero faster than their neighboring sides or if the maximum angle of the element approaches $\pi$), however this kind of quadrilateral cannot become too narrow. Condition  (\ref{innerangles}) will play an important role in the sequel and therefore we introduce the following alternative definition. 
\begin{defi}
\label{defidac}
We say that a quadrilateral $K$ satisfies the {\slshape{double angle condition}} with constants $\psi_m, \psi_M$, or shortly $DAC(\psi_m,\psi_M)$, if $K$ verifies $mac(\psi_m)$ and $MAC(\psi_M)$ simultaneously, i.e., if all inner angles $\theta$ of $K$ verify $0<\psi_m \leq \theta \leq \psi_M<\pi$.
\end{defi}
The   $DAC$  allows anisotropic quadrilaterals (such as narrow rectangles) as well as families of quadrilaterals that may degenerate into triangles. To see that consider, for instance, a quadrilateral with vertices $(0,0), (1,0), (s,1-s)$ and $(0,1-s)$ and take $0<s\to 0$.

For anisotropic quadrilaterals several papers have been written mainly in the \emph{isoparametric} case with $k=1$. In \cite{ZV1,ZV2} narrow quadrilaterals are studied and the authors require  the two longest sides of the element to be opposite and almost parallel, the constant C obtained by them  depends on an angle which \emph{in some cases} is the minimum angle of the element. \emph{Anisotropic error estimates} for small perturbations of rectangles have been derived in \cite{A1,A2}.  On the other hand, for $k=1$, more general and subtle conditions can be found in the literature. For $k=1$ and $p=2$, it is proved in \cite{AD2}  that the optimal error estimate (\ref{errorinwp}) can be obtained under the following weak condition 
 \begin{defi}
Let $K$ be a convex quadrilateral, and let $d_1$ and $d_2$ be the diagonals of $K$. We say that $K$ satisfies the {\slshape{regular decomposition property}} with constants $N \in \R$ and $0<\psi_M < \pi$, or shortly $RDP(N,\psi_M)$, if we can divide $K$ into two triangles along one of its diagonals, that will be called always $d_1$, in such a way that $|d_2|/|d_1| \leq N$ and both triangles have its maximum angle bounded by $\psi_M$.
\end{defi}

In Remarks 2.3 - 2.7 of \cite{AD2} it is shown that the regular decomposition property $RDP$ is certainly much weaker than those considered in previous works (including  \cite{A1,A2,CR,J2,ZV1,ZV2}). We collect for further reference some elementary remarks

\begin{rem}
\label{rem:regimplrdp}
If a quadrilateral $K$ satisfies \eqref{regcond} then  $K$ verifies $RDP(\sigma, \psi)$ whith $\psi=\psi(\sigma)$. Indeed, \eqref{regcond} implies that $K$ verifies $mac(\theta)$ for some $\theta=\theta(\sigma)>0$. Therefore
there is at most one angle of $K$ not bounded by $\pi -\theta$.
Dividing $K$ by the diagonal $d_1$ containing the vertex associated to that angle we get that $K$ satisfies $RDP(\sigma,\pi -\theta)$. 
\end{rem}
\begin{rem}
\label{rem:macimplrdp}
If a quadrilateral $K$ satisfies $MAC(\psi_M)$ then  $K$ verifies $RDP(1, \psi_M)$, as one can see by taking
$d_1$ as the longest diagonal of $K$.
\end{rem}
Since, by definition, $DAC(\psi_m,\psi_M)$ implies
$MAC(\psi_M)$ we have
\begin{rem}
\label{rem:dacimplrdp}
If a quadrilateral $K$ satisfies $DAC(\psi_m,\psi_M)$ then  $K$ verifies $RDP(1, \psi_M)$. 
\end{rem}
\begin{rem}
\label{dacorrdp} If $K$ verifies the $mac(\psi_m)$, then $K$  either verifies
\begin{enumerate}
 \item $DAC(\psi_m, \pi-\frac{\psi_m}{2})$ or
 \item the regularity condition \eqref{regcond} with $C=C(\psi_m)$.
\end{enumerate}
Indeed, assume that (1) does not hold. Then $K$ has an internal angle which is greater than $\pi-\psi_m/2$, it is easy to see that this angle is unique so we can call it $\theta$.  Divide $K$ into two triangles $T_1$ and $T_2$ through the diagonal opposite to $\theta$ in such a way that $\theta$ becomes an internal angle of $T_1$. Calling $\beta_1$ and $\beta_2$ to the other angles of $T_1$ it follows that $\beta_i < \psi_m/2$ with $i=1,2$. Let $\gamma_i$, $i=1,2$, the complementary angle of $\beta_i$ (w.r.t. the corresponding internal angle of $K$). It is easy to see that $\gamma_i>\psi_m/2$, so $T_2$ is a triangle that have its three internal angles bounded away from $0$. To be more precise, $T_2$ verifies $mac(\psi_m/2)$, therefore $T_2$ is a regular triangle in the sense of \eqref{regcond}
with $C=C(\psi_m/2)$. From this fact follows easily that $K$ is a regular quadrilateral in the same sense, i.e., in such a way that (\ref{regcond}) holds (actually the same constant $C$ can be used). \Box  
\end{rem}

\begin{rem}
Combining Remarks \ref{rem:regimplrdp},  \ref{rem:dacimplrdp} and \ref{dacorrdp} it is clear that
$mac \implies RDP$. 
\end{rem}
An strikingly result is that the $RDP$, in spite of being appropriate for $p=2$, does not work for arbitrary values of $p$. Indeed in \cite{AM} and for $k=1$ the results of \cite{AD2} are extended for the error in $W^{1,p}$ with  $1 \leq p <3$. Moreover, it is  shown, by means of a counterexample, that this range for $p$ is sharp. As a consequence the stronger $DAC$ (i.e. (\ref{innerangles})) is proposed and shown that under this condition the error estimate hold for all $p \geq 3$. The second result given in the present paper is that, for $p$ in this range, the $DAC$ is also a sufficient  for any  $k\ge 2$.
 
For the reader's convenience we summarize in the following table simple and sufficient geometric 
 conditions pointing out the role of $p$ and $k$ 
 \begin{center}
  \begin{tabular}{| l | c | r|}
    \hline
        & $1\le p<3$ & $3\le p$ \\ \hline
    $k=1$ & RDP & DAC \\ \hline 
    $k\ge 2$ & mac & DAC \\
    \hline
  \end{tabular}
\end{center}
the first
 row of the table was proved in \cite{AM} while
the new results are given in the second row.

To finish this short review  
we recall that for $k=1$ more results are available. In \cite{MN}, $H^1$ error estimates are obtained for the $\q_1$ isoparametric Lagrange interpolation under a weaker condition than the $RDP$. This condition can be regarded as a generalization of the last one and therefore called $GRDP$.   

 \bigskip
This paper is structured as follows: in Section $2$ we introduce a family of reference elements appropriate for dealing with general convex quadrilaterals and some key results are provided. In Section $3$ and Section $4$ our family is related to different geometric conditions (such as $RDP$, $MAC$, $DAC$ and $mac$) while some properties about the distribution of the interpolation nodes of the family are studied.
Section $5$ gives the general approach for bounding the interpolation error. Finally the main results as well as the counterexamples can be found in the last section of the paper.

\section{The family of reference $K(a,b,\at,\bt)$}
\label{section2}
\setcounter{equation}{0}
With $K$ we denote an arbitrary  \emph{convex} quadrilateral with vertices $V_1, V_2, V_3, V_4$ enumerated in  counterclockwise order. For positive numbers $a,b,\at,\bt$, we use $K(a,b,\at,\bt)$ to represent a  quadrilateral (always convex) with vertices $V_1=(0,0)$, $V_2=(a,0)$, $V_3=(\at,\bt)$ and $V_4=(0,b)$. In particular, $\hk=K(1,1,1,1)$ is the reference unit square and for any positive integer  $k$ we consider  $(k+1)^2$ points $\{ \hm_{ij} \}_{0 \leq i,j \leq k}$ of coordinates $\hx=j/k$ and $\hy=i/k$. For $\hk$ we write $\hat{V}_1=\hm_{00}$, $\hat{V}_2=\hm_{0k}$, $\hat{V}_3=\hm_{kk}$ and $\hat{V}_4=\hm_{k0}$.

We define as usual  $F_K : \hk \to K$  as $F_K(\hx)= \sum_{i=1}^4 V_i \hphi_i(\hx)$ being $\hphi_i$ the bilinear basis function associated with the vertex $\hat{V}_i$, i.e., $\hphi_i(\hat{V}_j)=\delta_i^j$. Then, in $K$ we have $(k+1)^2$ points $\{ M_{ij} \}_{0 \leq i,j \leq k}$ defined by
$$
M_{ij}=F_K(\hm_{ij}).
$$

For quadrilateral elements (isoparametric  when $k=1$ and subparametric otherwise), we have the basis function on $K$ defined by $\phi_{ij}(X)=\hphi_{ij}(F_K^{-1}(X))$ where $\hphi_{ij} \in \widehat{\q}_k(\hk)$ verifies $\hphi_{ij}({\hm_{lr}})=\delta_{ij}^{lr}$ and therefore the $\q_k$ interpolation operator $Q_k$, on $K$ is given by
$$Q_k u(X)=\hQ_k \hu(\widehat{X})$$
where $X=F_K(\widehat{X})$ and $\hQ_k$ is the Lagrange interpolation of order $k$ of $\hu = u \circ F_K$ on $\hk$. Interpolation nodes of the form $\{ M_{ij} \}_{1 \leq i,j \leq k-1}$ are called \emph{interior} and any $M_{ij}$ which is not an interior node is called an \emph{edge} node.
Also of interest is the triangle  $T(a,b)$ of vertices $V^T_1=(0,0), V^T_2=(a,0),V^T_3=(0,b)$.
The interpolation nodes $M_{ij}^T$ of the Lagrange interpolation operator  $\Pi_k\in \mathbb{P}_k$ of degree $k$ are given by $M_{ij}^T=(aj/k,bi/k)$, $0\le i+j\le k$.
With $C$ we denote a positive constant that may change from line to line. Sometimes we also use the notation $x\comp y$  for   positive variables if they are comparable in the following sense $\frac{1}{C}x \le y\le Cx$.

For any element of the type $K(a,b,\at,\bt)$ considered in this work the following condition will become relevant
in different contexts
\begin{equation}
 \label{condDelta1}
 \mbox{Condition ($\Delta$1): } \qquad \frac{\at}{a}, \frac{\bt}{b} \leq C.
\end{equation}
This condition takes sometimes the more restrictive form
\begin{equation}
 \label{condD1}
 \mbox{Condition ($D$1): } \qquad \frac{\at}{a}, \frac{\bt}{b} \leq 1.
\end{equation}
In spite of the fact that both \eqref{condDelta1} and \eqref{condD1} look similar they characterize,
under the supplementary geometric assumption \eqref{condD2} (see below),  different classes of elements.

Calling $d_1$ to the diagonal joining $V_2$ and $V_4$ we see that $d_1$  divides
$K$ into two triangles, that we call $T_1$ and $T_2$ (see Figure \ref{fig:Kab}). For the angle $\alpha$  of $T_1$ placed at $V_4$, we  introduce 
\begin{equation}
 \label{condD2}
 \mbox{Condition ($D$2): } \qquad \frac{1}{sin(\alpha)} \leq C,
\end{equation}
which says that $\alpha$ is bounded away from $0$ and $\pi$. 

\begin{figure}[h]
\resizebox{8cm}{6cm}{\includegraphics{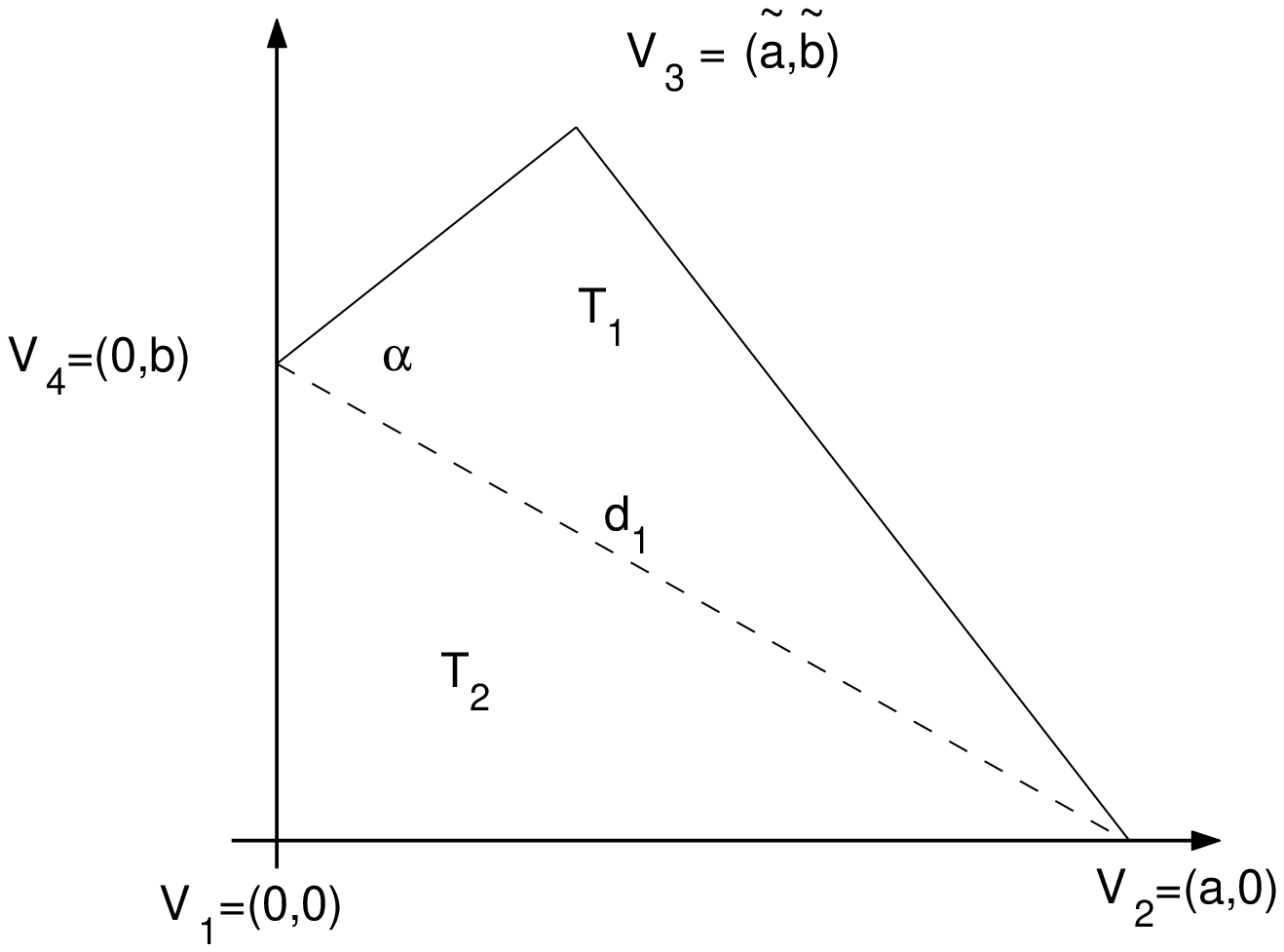}}
\caption{Notation for an element $K(a,b,\at,\bt)$.}
 \label{fig:Kab}
\end{figure}
Finally, in order to exploit some results
given in previous works, we introduce yet another 
useful condition
\begin{equation}
 \label{condDelta2}
 \mbox{Condition ($\Delta$2): } \qquad  |l| \leq C |s|.
\end{equation}
where $l$ is the segment $\overline{V_3V_4}$ joining $V_3$ and $V_4$ and $s$ denotes the shortest side of $K(a,b,\at,\bt)$. That is, $(\Delta 2)$ amounts to say that the side $l$ is comparable to the shortest side of $K$. 

Not difficult to prove is the following
\begin{lemma}
 \label{lema:otracond}
 Let $K(a,b,\at,\bt)$ a general convex quadrilateral. Then, conditions 
 $[\Delta 1,D2]$  and  $[\Delta 2, D2]$
are equivalents.

\end{lemma}
\begin{proof}
That $(\Delta 2)$ implies $(\Delta 1)$ follows
easily as we have $\at\le |l|\le C |s|\le C a$ and similarly $|b-\bt|\le Cb$. Hence triangular inequality yields $\bt\le Cb$. On the other hand
assume $[\Delta 1,D2]$. Thanks to $(D2)$, and using
the law of sines for the triangle $\Delta(V_2,V_3,V_4)$ we see that the angle $\beta$ at $V_3$ is away from
zero and $\pi$. Indeed $\frac{\sin \alpha}{\sin \beta }=\frac{|V_2V_3|}{|d_1|} \le C$ due to $(\Delta 1)$. Since
both $\alpha$ and $\beta$ are away from zero and $\pi$ the law of sines says that actually $d_1$ is comparable to $|V_2V_3|$.
It implies, in turn, that $|l|\le C \min \{|V_2V_3|,|d_1| \} $. Now consider two cases: 
if $\frac{b}{a}$ approaches zero then necessarily $\frac{|b-\bt|}{\at}>C>0$, otherwise the angle $\alpha$ can not obey $(D2)$. Hence we have simultaneously, $b<a$ 
and (due to $(\Delta 1)$) $|l|=\sqrt{(b-\bt)^ 2+\at^2}\le C|b-\bt|\le Cb$  showing $(\Delta 2)$. To finish, let us assume that $\frac{b}{a}>C>0$,
in such a case if $a$ and $b$ are comparable we have nothing to prove since then they are also comparable to $d_1$. Therefore we may assume that $\frac{a}{b}$ approaches cero. In that case $\frac{|b-\bt|}{\at}\le C$ since $\alpha$ can not get close to $\pi$ nor to zero. Therefore we have $a<b$ and
$|l|= \sqrt{(b-\bt)^ 2+\at^2}\le C\at \le Ca$ thanks to $(D1)$. The lemma follows. \Box
\end{proof}

The class of reference elements of the type $K(a,b,\at,\bt)$ is adequate for dealing with
several geometrical conditions. For instance,
as we show below, [$\Delta 1,D2$] (or equivalently [$\Delta 2,D2$]) describe any element under the RDP and $[D1,D2]$ the family of elements under the $DAC$, etc.

In order to do so, the key tool is given by the following elementary lemma 

\begin{lemma}
\label{comperror}
Let $K$, $\oK$ be two convex quadrilateral elements, and let $L:K \to \oK$ be an affine transformation $L(X)=BX+P$. Assume that $L(K)=\oK$, $\|B\|,\|B^{-1}\| < C$ (in particular the condition number $\kappa(B)< C$). If $\oQ_k$ is the $\mathcal{Q}_k$ (isoparametric
for $k=1$ or subparametric for $k>1$) interpolation on $\oK$ and $\ou=u \circ L^{-1}$ then for any $p \geq 1$
$$C_1 |\ou-\oQ_k\ou|_{1,p,\oK} \leq |u-Q_ku|_{1,p,K} \leq C_2 |\ou-\oQ_k\ou|_{1,p,\oK}$$
and
$$C_1 |\ou|_{m,p,\oK} \leq |u|_{m,p,K} \leq C_2 |\ou|_{m,p,\oK}$$
for any $m \geq 1$.
\end{lemma}

\begin{proof} By definition of the interpolation we have
$$Q_ku(X)=\hQ_k \hu(F_K^{-1}(X)) \quad \mbox{and} \quad \oQ_k \ou(\overline{X})=\hQ_k \hat{\ou}(F_{\oK}^{-1}(\overline{X})).$$

Where $\overline{X}$ denotes the variable on $\oK$. Since $L$ is an \emph{affine} transformation, $F_{\oK} = L \circ F_K$ and so $\oQ_k \ou(\overline{X})=Q_ku(X)$. Then, the lemma follows easily by observing that $\left\| B \right\|, \left\| B^{-1} \right\| < C$. \Box
\end{proof}
\begin{defi}
 \label{def:equi}
 We say that two quadrilateral elements $K,\oK$ are $C$-equivalents (or simply \emph{equivalents}) if and only if they
can be mapped to each other by means of  an affine transformation of the kind given in Lemma \ref{comperror}.  
\end{defi}
Taking into account that each geometric condition defined as far is going to be mapped to an appropriate class of  \emph{equivalent} elements $K(a,b,\at,\bt)$ it is important to consider the map $F_K:\hat K \to K(a,b,\at,\bt)$,
\bq
\label{FK}
F_K(\hx,\hy)=(a\hx(1-\hy)+\at\hx\hy, b\hy(1-\hx)+\bt\hx\hy)=(x,y),
\eq
as well as its associated Jacobian  

\bq
\label{DFK}
DF_K(\hx,\hy) =
\left( \begin{array}{cc}
a+\hy(\at-a) & \hx(\at-a)\\
\hy(\bt-b) & b+\hx(\bt-b)
\end{array} \right),
\eq

\bq
\label{jacoFK}
J_K:=\det(DF_K)(\hx,\hy)=ab(1+\hx(\bt/b-1)+\hy(\at/a-1)).
\eq

Observe that since $K$ is convex, we have $J_K>0$ in the interior of $\hk$. Indeed, since $J_K$ is an affine function it is enough to verify that it is positive at some vertex of $\hk$ and non negative at the remaining ones. The positivity at $\hat{V}_1=(0,0)$ is trivial, as well as the non negativity at $\hat{V}_2=(a,0)$ and $\hat{V}_4=(0,b)$. On the other hand, since $K$ is convex, $(\at,\bt)$ lies above the segment joining $V_2$ and $V_4$ (for this segment $y(x)=-(b/a)(x-a)$ and $0<\frac{\bt-y(\at)}{b}=\frac{\bt}{b}+\frac{\at}{a}-1$) therefore,  
\bq
\label{convJK}
J_K(1,1)=ab(\at/a+\bt/b-1) > 0.
\eq

Following again \cite{AD2,AM} we introduce for  $p\ge 1$  the next expression that becomes useful in the sequel
 \bq
\label{defI}
I_p=I_p(a,b,\at,\bt):= \ds \int_0^1 \int_0^1 \frac{1}{(1+\hx(\bt/b-1)+\hy(\at/a-1))^{p-1}}\ d\hx d\hy,
\eq
 where the numbers $a,b,\at,\bt$ are compatible with an element $K(a,b,\at,\bt)$.

\begin{lemma}
\label{edgesnodes}
If $K=K(a,b,\at,\bt)$ is convex and  $(\Delta 1)$ given by \eqref{condDelta1} holds, then for any $p \geq 1$ and for any function basis $\phi$ there exists a positive constant $C$ such that
\bq
\ds \left\| \frac{\partial\phi}{\partial x} \right\|_{0,p,K}^p \leq C \frac{b}{a^{p-1}} I_p
\eq

\bq
\ds \left\| \frac{\partial\phi}{\partial y} \right\|_{0,p,K}^p \leq C \frac{a}{b^{p-1}} I_p.
\eq
\end{lemma}

\begin{proof} Let $\hphi$ the function basis on $\hk$ corresponding to $\phi$, then (from the chain rule) follows that
$$\left( \begin{array}{c}
\ds \left(\frac{\partial \phi}{\partial x} \circ F_K \right) (\hx,\hy)\\
\ds \left(\frac{\partial \phi}{\partial y} \circ F_K \right) (\hx,\hy)
\end{array} \right) =
\frac{1}{J_K(\hx,\hy)}
\left( \begin{array}{ll}
b+\hx(\bt-b) &
-\hy(\bt-b)\\
\\
-\hx(\at-a) &
a+\hy(\at-a)
\end{array} \right)
\left( \begin{array}{c}
\ds \frac{\partial \hphi}{\partial \hx}(\hx,\hy)\\
\\
\ds \frac{\partial \hphi}{\partial \hy}(\hx,\hy)
\end{array} \right)$$

where $J_K(\hx,\hy)=\det(DF_K)=ab(1+\hx(\bt/b-1)+\hy(\at/a-1))$.

Calling $R(\hx,\hy)=(1+\hx(\bt/b-1))\frac{\partial \hphi}{\partial \hx}(\hx,\hy)-\hy(\bt/b-1)\frac{\partial \hphi}{\partial \hy}(\hx,\hy)$ and $S(\hx,\hy)=-\hx(\at/a-1))\frac{\partial \hphi}{\partial \hx}(\hx,\hy)+(1+\hy(\at/a-1))\frac{\partial \hphi}{\partial \hy}(\hx,\hy)$ we have
$$\ds \left(\frac{\partial \phi}{\partial x} \circ F_K \right) (\hx,\hy) = \frac{b}{J_K(\hx,\hy)} R(\hx,\hy)
\quad \mbox{ and }  \quad
\left(\frac{\partial \phi}{\partial y} \circ F_K \right) (\hx,\hy) = \frac{a}{J_K(\hx,\hy)} S(\hx,\hy).$$

By changing variables we get
$$\ds \left\| \frac{\partial \phi}{\partial x} \right\|_{0,p,K}^p =
\frac{b}{a^{p-1}} \int_0^1 \int_0^1  \frac{\left| R(\hx,\hy) \right|^p}{(1+\hx(\bt/b-1)+\hy(\at/a-1))^{p-1}}\ d\hx d\hy$$
and
$$\ds \left\| \frac{\partial \phi}{\partial y} \right\|_{0,p,K}^p =
\frac{a}{b^{p-1}} \int_0^1 \int_0^1 \frac{\left| S(\hx,\hy) \right|^p}{(1+\hx(\bt/b-1)+\hy(\at/a-1))^{p-1}}\ d\hx d\hy,$$

and the proof concludes using that $R$ and $S$ are uniformly bounded since they are polynomials, $0 \leq \hx,\hy \leq 1$ and $0 \leq \at/a,\bt/b \leq C$ by ($\Delta$1). $\Box$
\end{proof}

Previous result provides bounds for any basis function. As we show later basis functions associated to  \emph{internal nodes} require a particular treatment. In particular we have,  

\begin{lemma}
\label{derparc}
If $K=K(a,b,\at,\bt)$ is convex and $(\Delta 1)$ given by \eqref{condDelta1} holds then for any $p \geq 1$ and for any function basis $\phi$ associated to an internal node of $K$, there exists a positive constant $C$ such that
\bq
\label{eq:prima}
\ds \left\| \frac{\partial\phi}{\partial x} \right\|_{0,p,K}^p \leq C \frac{b}{a^{p-1}}
\left[ |1-\bt/b|^p I_p+max\{1,(b/\bt)^{p-1}/2\} \right]
\eq

\bq
\label{eq:seconda}
\ds \left\| \frac{\partial\phi}{\partial y} \right\|_{0,p,K}^p \leq C \frac{a}{b^{p-1}}
\left[ (\at/a)^p I_p+max\{1,(b/\bt)^{p-1}/2\} \right].
\eq
\end{lemma}

\begin{proof} Since $\phi$ is associated to an internal node on $K$, it follows that $\hphi$ is associated to an internal node on $\hk$ so that there exists $P \in \q_{k-2}(\hk)$ such that $\hphi(\hx,\hy)=\hx(1-\hx)\hy(1-\hy)P(\hx,\hy)$. Then
$$\frac{\partial \hphi}{\partial \hx}=\hy(1-\hy)A \quad \mbox{ and } \quad
\frac{\partial \hphi}{\partial \hy}=\hx(1-\hx)B$$

where $A(\hx,\hy)=\frac{\partial}{\partial \hx} \left[ \hx(1-\hx)P(\hx,\hy) \right]$ and $B(\hx,\hy)=\frac{\partial}{\partial \hy} \left[ \hy(1-\hy)P(\hx,\hy) \right]$.

From the chain rule follows that
$$\left( \begin{array}{c}
\ds \left(\frac{\partial \phi}{\partial x} \circ F_K \right) (\hx,\hy)\\
\ds \left(\frac{\partial \phi}{\partial y} \circ F_K \right) (\hx,\hy)
\end{array} \right) =
\frac{1}{J_K(\hx,\hy)}
\left( \begin{array}{ll}
b+\hx(\bt-b) &
-\hy(\bt-b)\\
\\
-\hx(\at-a) &
a+\hy(\at-a)
\end{array} \right)
\left( \begin{array}{c}
\hy(1-\hy)A(\hx,\hy)\\
\\
\hx(1-\hx)B(\hx,\hy)
\end{array} \right)$$

where $J_K(\hx,\hy)=\det(DF_K)=ab(1+\hx(\bt/b-1)+\hy(\at/a-1))$.

Calling $S=\hx\hy[(1-\hx)B-(1-\hy)A]$ and $R=\hy A$ we have
$$\ds \left(\frac{\partial \phi}{\partial x} \circ F_K \right) (\hx,\hy)=
\frac{b}{J_K}\left[ (1-\bt/b)S+(1-\hy)R \right]$$

and by a change of variables we get
$$\ds \left\| \frac{\partial \phi}{\partial x} \right\|_{0,p,K}^p =
\int_0^1 \int_0^1 \frac{b}{a^{p-1}} \frac{\left| (1-\bt/b)S+(1-\hy)R \right|^p}{(1+\hx(\bt/b-1)+\hy(\at/a-1))^{p-1}}\ d\hx d\hy.$$

Using the fact that $S$ and $R$ are uniformly bounded we see that 
$$\ds \left\| \frac{\partial \phi}{\partial x} \right\|_{0,p,K}^p \leq
C \frac{b}{a^{p-1}} \int_0^1 \int_0^1 \frac{|1-\bt/b|^p+(1-\hy)^p }{(1+\hx(\bt/b-1)+\hy(\at/a-1))^{p-1}}\ d\hx d\hy.$$

Then
$$\ds \left\| \frac{\partial \phi}{\partial x} \right\|_{0,p,K}^p \leq
C \frac{b}{a^{p-1}} \left[ |1-\bt/b|^pI_p + \int_0^1 \int_0^1 \frac{(1-\hy)^p }{(1+\hx(\bt/b-1)+\hy(\at/a-1))^{p-1}}\ d\hx d\hy \right].$$

From the convexity of $K$ we have $\at/a+\bt/b-1>0$, hence
$$1+\hx(\bt/b-1)+\hy(\at/a-1) > 1+\hx(\bt/b-1)-\hy\bt/b$$
Assume now that $\bt/b<1$. Since $0 \leq \hx \leq 1$ we conclude $1+\hx(\bt/b-1) \geq \bt/b$ and finally
$$1+\hx(\bt/b-1)+\hy(\at/a-1) \geq \bt/b(1-\hy).$$
Therefore
$$\ds \int_0^1 \int_0^1 \frac{(1-\hy)^p }{(1+\hx(\bt/b-1)+\hy(\at/a-1))^{p-1}}\ d\hx d\hy \leq
\frac{1}{2} \left( \frac{b}{\bt} \right)^{p-1}.$$
On the other hand, if $\bt/b\ge 1$ 

$$1+\hx(\bt/b-1)+\hy(\at/a-1) \ge 
1+\hy(\at/a-1)\ge 1-\hy$$
hence
$$\ds \int_0^1 \int_0^1 \frac{(1-\hy)^p }{(1+\hx(\bt/b-1)+\hy(\at/a-1))^{p-1}}\ d\hx d\hy \leq
1,$$
and \eqref{eq:prima} follows.

Finally, the estimate for $\left\| \frac{\partial \phi}{\partial y} \right\|^p_{0,p,K}$ can be obtained in a similar way from the expression

$$\ds \left(\frac{\partial \phi}{\partial y} \circ F_K \right) (\hx,\hy)=
\frac{a}{J_K}\left[ \at/a S+(1-\hy)\bar{R} \right]$$

where $\bar{R}=\hx(1-\hx)B+\hx R$. $\Box$
\end{proof}
In order to clarify in advance the role of the
term $\frac{b}{\bt}$ in the previous lemma
let us notice the following 
\begin{lemma}
 \label{lema:cortoarriba}
 For any arbitrary convex
 quadrilateral $K(a,b,\at,\bt)$  under condition $[\Delta 1,D2]$ (equiv. $[\Delta 2,D2]$)   there exists another \emph{equivalent}
 element (in the sense of Definition \ref{def:equi}) obeying $[\Delta 1,D2]$ (equiv. $[\Delta 2,D2]$, see Lemma \ref{lema:otracond}) with the same constants and for which  $\frac{\bt}{b}\ge \frac{1}{2}$.
\end{lemma}
\begin{proof} Consider the triangle $V_2V_3V_4$ and the angles $\alpha$ and $\beta$ at $V_4$
and $V_2$ respectively. If for the original $K(a,b,\at,\bt)$,   $\frac{\bt}{b}\ge \frac{1}{2}$, then we have nothing to prove. Otherwise $\frac{\bt}{b}< \frac{1}{2}$ and hence we see that $\alpha\le \beta$. On the other hand, since both are interior angles of a triangle, $\beta\le \pi-\alpha$, therefore using $(D2)$ we see that $\beta$ is away from $0$ and $\pi$ and therefore under a rigid movement we can transform our element
into $K(b,a,\bt,\at)$. The resulting element satisfies the required conditions $[\Delta 1,D2]$ with the same constants than those $K(a,b,\at,\bt)$ and the lemma follows thanks to the fact that $\frac{\at}{a}\ge \frac12$ (see \eqref{convJK}). 
 \Box
\end{proof}
\begin{rem}
 \label{rem:alsodac}
 Observe that previous lemma also applies to
 elements $K(a,b,\at,\bt)$ under $[D1,D2]$ since
 $D1\implies \Delta 1$.
\end{rem}

\section{$K(a,b,\at,\bt)$ and different geometric conditions}
In this section we explore in detail how to use the family $K(a,b,\at,\bt)$. The following 
lemma is useful in the rest of this section.
\begin{lemma}
 \label{lema:angles}
 Let $L$ be the linear transformation associated with a matrix $B$. Given
two vectors $w_1$ and $w_2$, let $\alpha$ be the angle between them and let $\overline\alpha$ be the angle between $L(w_1)$ and $L(w_2)$. 
Calling $\kappa(B)$ the condition number of $B$ we have
$$
\frac{2}{\kappa(B)\pi}\alpha\le \overline \alpha \le \pi \left(1-\frac{2}{\kappa(B)\pi}\right)+\alpha \frac{2}{\kappa(B)\pi}.
$$
\end{lemma}
\begin{proof} The proof is elementary and can be found in \cite{AD1}. \Box
\end{proof}

\subsection{The $RDP$ and the family $K(a,b,\at,\bt)$} 
In order to characterize the elements under the $RDP$ we begin with the following elementary 
result 
\begin{lemma}
\label{lema:losRDP}
 Let $K$ be an element of the type
 $K(a,b,\at,\bt)$  and assume $[\Delta 1, D2]$
 (equivalently $[\Delta 2, D2]$). Then $K$ verifies the $RDP$ with constants depending only on those given in conditions $[\Delta 1, D2]$.
\end{lemma}
\begin{proof}
 Follows straightforwardly taking $d_1=\overline{V_2V_4}$ as the dividing diagonal. \Box
\end{proof}

Now we are ready for the following characterization

\begin{theorem}
 \label{teo:carrdp}
Let $K$ be a general convex quadrilateral. Then
$K$ verifies the $RDP$ if and only if  $K$ is \emph{equivalent} to some $K(a,b,\at,\bt)$ under
$[\Delta 2, D2]$ (equiv. $[\Delta 1, D2]$).
  \end{theorem}
\begin{proof}
First we assume that $K$ is equivalent to some 
$K(a,b,\at,\bt)$ under
$[\Delta 2, D2]$. From Lemma \ref{lema:losRDP}
we know that $K(a,b,\at,\bt)$ verifies $RDP(N,\psi_M)$ with constants bounded in terms of those given in $[\Delta 2, D2]$. Since
$K=L(K(a,b,\at,\bt))$ for certain affine mapping $Lx=Bx+P$ of the type considered in the Definition \ref{def:equi} we see from 
Lemma \ref{lema:angles} and taking  into account that such an $L$ preserves lengths (up to a constant depending on $\|B\|,\|B^{-1}\|<C$) that
$K$ verifies the $RDP$ with constants depending on $L$ as well as the $RDP$ constants associated to $K(a,b,\at,\bt)$. To show the other implication  we follow \cite{AD2}. Assume that  $K$ satisfies the $RDP$ and divide it along $d_1$ into $T_1$ and $T_2$  in such a way that all their interior angles are bounded by $\psi_M$, while $\frac{|d_2|}{|d_1|}\le N$. We choose the notation in such a way that the shortest side of $K$, called $s$, is one of the sides
of $T_1$ and call $\beta$ the angle of $T_2$ opposite to $d_1$ . After a rigid movement we can assume
that the vertex $V_1$ corresponding to $\beta$ is placed at the origin and that $K$ is contained in the upper half-plane. We can also assume that 
$V_2$ is placed at the point $(a,0)$ with $a>0$,
being the side $\overline{V_1V_2}$ opposite to the shortest side of $K$. Define now $b=|\overline{V_1V_4}|\sin(\beta)$. Then, we have that $V_4$ is placed at $(\cot( \beta) b, b)$. Let us notice that  $\beta$ is away from $\pi$, since $\beta <\psi_M$. Moreover, it is also away from $0$ as one can see by means of the law of sines and taking into account that  $d_1$ (the side of $T_2$ opposite to $\beta$) is comparable to the largest side of $T_2$ (due to the fact that $|d_2|\le N |d_1|$ and recalling that the diameter of $K$  agrees with the length of the longest diagonal).  
Then the linear mapping $L$
associated to the matrix $B=\begin{pmatrix}
1 & \cot(\beta)\\
0&1 
\end{pmatrix}$ performs the desired transformation $L(K(a,b,\at,\bt))=K$ while it fulfills the requirements of Definition \ref{def:equi} as $\|B\|,\|B^{-1}\|<C$ (with
$C$ depending on $\psi_M,N$).
In particular  $\kappa(B)<\frac{2}{\sin^2(\beta)}$. On the other hand, calling $L(\at,\bt)=V_3$ we observe that $(\Delta 2)$ holds
since $L$ preserves lengths (up to constants depending on  $\|B\|,\|B^{-1}\|$ and $\overline{V_3V_4}=L(s)$. On the other hand, since $T_1$ verifies $MAC(\psi_M)$ and $s$ is the
shortest side of $T_1$ then the angle of $T_1$ placed at the common vertex of $d_1$ and $s$ is away from $0$ and $\pi$. Therefore $(D2)$ holds
thanks to Lemma \ref{lema:angles}. The theorem follows.   \Box
 \end{proof}
From now on (see Lemma \ref{comperror}) we assume
that any element verifying the $RDP$ is of the kind  $K(a,b,\at,\bt)$ under $[\Delta 2, D2]$ (equiv. $[\Delta 1, D2]$). In \cite{AM} it is proved that the $RDP$ is sufficient to get optimal order  error estimates in $W^{1,p}$ \emph{for} $Q_1$ whenever $1\le p <3$. 
In the last section we give a counterexample showing in particular that this result does not hold for $k\ge 2$. 

The next result, borrowed from \cite{AM},  help us to shorten our exposition playing also a role in the construction of a counterexample. 

\begin{lemma}
\label{lemaI}
Let $K=K(a,b,\at,\bt)$ a convex quadrilateral. Assume $[\Delta 2, D2]$ (equiv. $[\Delta 1, D2]$), then for any $1\le p < 3$
\bq
\label{I}
\ds \max \left\{ \frac{a}{b^{p-1}}, \frac{b}{a^{p-1}} \right\} I_p \leq C \frac{h}{|l|^{p-1}}
\eq

with $C$ a constant depending only on those given in  $[\Delta 2, D2]$ and $p$.
\end{lemma}

\proof 
See the proof of Lemma 3.5  in \cite[pag. 140]{AM} together with eqs. (15) and (16) in that paper. 
In the mentioned lemma  it is precisely the expression \eqref{I} what is proved.  $K(a,b,\at,\bt)$ and $I_p$ have the same meaning that
in the present work (see also \cite[pag. 136]{AM}, where the invoked hypotheses (H1),(H2),(H3) and (H4) are introduced and derived from the RDP condition.)  $\Box$\\
 
\begin{rem}
\label{rem:p<3}
 Although it is not written here we know from \cite{AM} that the constant $C$ in
 \eqref{I} may behave like $1/(3-p)$. In order
 to get an uniform bound for $3\le p$ it is necessary to restrict the class of the underlying reference elements $K=K(a,b,\at,\bt)$. Later we show 
 that \eqref{I} holds for any $1\le p$ if we work with the family $K(a,b,\at,\bt)$ associated to the $DAC$.  
\end{rem}

\subsection{The regularity condition $h/\rho<\sigma$ and the family $K(a,b,\at,\bt)$} 
For dealing with regular elements we need to introduce a new geometrical condition associated to the class $K(a,b,\at,\bt)$,
\begin{equation}
 \label{condD3}
 \mbox{Condition ($D$3): } \qquad a\comp b.
\end{equation} 
 
\begin{theorem}
\label{teo:carreg}
 Let $K$ be a general convex quadrilateral. Then $K$ is regular (in the sense of \eqref{regcond}) if and only 
 it is equivalent to  some $K(a,b,\at,\bt)$ under $[\Delta 2, D2, D3]$ (equiv. $[\Delta 1, D2, D3]$).
\end{theorem}
\begin{proof}
Thanks to Remark \ref{rem:regimplrdp} we know that
elements satisfying the regularity condition \eqref{regcond} satisfy also 
the $RDP$. Therefore from Theorem \ref{teo:carrdp} we see that $K$ can be  mapped, by means of an affine transformation $LX=BX+P$ (see Definition \ref{def:equi}) into an element $K(a,b,\at,\bt)$. This reference element should be regular  since $\|B\|,\|B^{-1}\|<C$ and now it is easy to see that $(\Delta 1)$ implies $(D3)$. To prove the other implication
take an element $K(a,b,\at,\bt)$ under $[\Delta 2, D2, D3]$ (equiv. $[\Delta 1, D2, D3]$). Thanks to $(D3)$ we have that $T(a,b)$ is a  regular
triangle. Now it is easy to see that this together with $[\Delta 1, D2]$ implies that $K(a,b,\at,\bt)$  verifies \eqref{regcond}. As a consequence,
any element of the form $L(K(a,b,\at,\bt))$ is regular for any affine mapping
of the kind considered in Definition \ref{def:equi}.  \Box
 \end{proof}
\subsection{The $DAC$ and the family $K(a,b,\at,\bt)$}
As it is mentioned before $DAC$ implies the $RDP$, and as consequence Theorem  \ref{teo:carrdp} says that any element under the  $DAC$ can be mapped into an element $K(a,b,\at,\bt)$ for which $[\Delta 1, D2]$ holds. Nevertheless, the following result, partially borrowed from \cite{AM}, states that actually we may assume $\frac{\at}{a}, \frac{\bt}{b} \leq 1$, and this, as we show later, not only simplifies the treatment of the error but also allows to deal with the case $1\le p$.  

\begin{theorem}
\label{teo:cardac}
Let $K$ be a general convex quadrilateral.  Then $K$ satisfies the $DAC(\psi_m,\psi_M)$ if and only if it is equivalent to an element $K(a,b,\at,\bt)$ under $[D1,D2]$.
\end{theorem}
\begin{proof}
Notice that it is always possible to select two neighboring sides $l_1,l_2$ of $K$ such that 
the parallelogram defined by these sides 
contains the element $K$. Call $V_1$ the common vertex of $l_1,l_2$ and $\beta$ the angle at $V_1$. After a rigid movement we may assume that $V_1=(0,0)$ and that $l_2$ lies along the $x$ axis (with nonnegative coordinates $(a,0)$). Moreover we can also assume that $l_1$ belongs to the upper half plane. Notice that  $l_{1}$ is the side joining $V_1V_4$ and following the proof of Theorem \ref{teo:carrdp} take  $b=|l_{1}| \sin(\beta)$ in such a way that $V_4$ can be written as $V_4=(b \cot(\beta),b)$. Then the linear mapping $L$ associated to the matrix $B$ defined in such theorem
 performs the desired transformation. Indeed, since  $\|B\|,\|B^{-1}\|<\frac{\sqrt{2}}{\sin(\beta)}$ with $\beta$ away from $0$ and $\pi$ (due to the fact that $K$ is under the $DAC$) we know that $L$ is of the class considered in Definition \ref{def:equi}. On the other hand, calling $L(\at,\bt)=V_3$ we observe that $(D1)$ holds  while thanks to the fact that the angle at $V_3$ is away from $0$ and $\pi$ the same holds for the angle at $(\at,\bt)$ meaning that at least one of the remaining angles of the triangle of vertices $(0,b),(\at,\bt),(a,0)$ does not approach zero nor $\pi$. Performing a rigid movement if necessary we may assume that this is the one at 
$(0,b)$ and hence $(D2)$ follows.  
Reciprocally, assume that $K(a,b,\at,\bt)$ verifies $[D1,D2]$ and it is equivalent to $K$. Notice that the maximal and minimal angle of $K(a,b,\at,\bt)$ are away from $0$ and $\pi$ (in terms of the constants given by $[D1,D2]$). Indeed, since at $V_1$ we have a right angle we only need to check the remaining vertices.  The angle at 
$V_4$ is bounded above by $\pi/2$ due to $(D1)$
and below by $\alpha$. Let us focus now on the angle at vertex $V_3$. It should be bounded below by $\pi/2$ due to $(D1)$. On the other hand, it can not approach $0$ due to $(D2)$. Finally, the angle at $V_2$ is greater than $\alpha$ and also bounded above by $\pi/2$. The proof concludes by using that $K$ is equivalent to $K(a,b,\at,\bt)$ (in the sense of Definition \ref{def:equi}) and Lemma \ref{lema:angles}.
\Box
\end{proof}

There is a property that can be derived from
$[D1,D2]$.
\begin{lemma}
 \label{lema:ladoacorto}
Let  $K(a,b,\at,\bt)$ be a general element under $[D1,D2]$, then $a\le Cb$.
\end{lemma}
\begin{proof}
The proof is elementary since $\tan \alpha\le \frac{b}{a}$. \Box
\end{proof}
An advantage of the $DAC$ is that it simplifies the treatment of $I_p$. Indeed
from $(D1)$ we get
$$
\frac{1}{(1+\hat x(\frac{\bt}{b}-1) +\hat y(\frac{\at}{a}-1))^{p-1}}\le 
\frac{1}{(\frac{\bt}{b}+\frac{\at}{a}-1)^{p-1}},
$$
since $0\le \hat x,\hat y\le 1$. On the other hand, calling $y(x)=-(b/a)(x-a)$
to the equation of the straight line joining $V_2$ and $V_4$ we have
$$
\frac{\bt-y(\at)}{b}=\frac{\at-y^{-1}(\bt)}{a}=\frac{\bt}{b}+\frac{\at}{a}-1,
$$
and since $|l| \sin(\alpha)\le \bt-y(\at)$ and $|l| \sin(\alpha)\le \at-y^{-1}(\bt)$
we get 
\begin{equation}
 \label{eq:IPintermsofsin}
I_p\le \min\{a^{p-1},b^{p-1}\}\frac{1}{|l|^{p-1}\sin(\alpha)^{p-1}}. 
\end{equation}
As a consequence we obtain for $DAC$ the following equivalent of Lemma \ref{lemaI}, that holds \emph{for any} $1\le p$.

\begin{lemma}
\label{lemaIparadac}
If $K=K(a,b,\at,\bt)$ is a general convex quadrilateral under $[D1,D2]$ then, for any $1\le p $,
\bq
\label{Iparadac}
\ds \max \left\{ \frac{a}{b^{p-1}}, \frac{b}{a^{p-1}} \right\} I_p \leq C \frac{h}{|l|^{p-1}}
\eq
with a constant  $C$  depending only on those of $[D1,D2]$.
\end{lemma}

\subsection{The $mac$ and the family $K(a,b,\at,\bt)$} We finish this section with a characterization of the elements under the minimal angle condition. A direct consequence
of previous subsections an Remark \ref{dacorrdp}
is the following

\begin{theorem}
 \label{teo:carmac}
 Let $K$ be a general convex quadrilateral.
 Then $K$ satisfies the $mac$ if and only
 $K$ is equivalent to some $K(a,b,\at,\bt)$
 for which holds either $[D1,D2]$ or $[\Delta 1, D2,D3]$ (equiv. $[\Delta 2, D2,D3]$). 
 \end{theorem}
\begin{proof}
Follows straightforwardly from Theorems \ref{teo:carreg} and \ref{teo:cardac} together with Remark \ref{dacorrdp}. \Box
\end{proof}

\section{Triangles under the $MAC$ inside $K(a,b,\at,\bt)$.}
\label{section:triang}

Let $K=K(a,b,\at,\bt)$ be a general convex quadrilateral and consider the associated triangle $T=T(a,b)$. Let us recall that $M_{ij}$ (resp. $M_{ij}^T$) are the interpolation nodes of 
$Q_k$ on $K$ (resp. $\Pi_k$ on  $T$). We are interested, loosely speaking, in the problem
of finding for each $M_{ij}$ a close enough $M_{ij}^T$.    Notice that in general
$M_{ij}$ does not agree with $M_{ij}^T$, except
for $i=0$ or $j=0$. For any other node $M_{ij}$ (i.e. for $i\neq 0 \neq j$) we consider
a suitable triangle having $M_{ij}$ as one of its vertices and with the
remaining vertices belonging to the set of (edge) interpolation nodes of $\Pi_k$.
We choose it in the following way: if $M_{ij}$ is an \emph{edge} node on the top
of $K(a,b,\at,\bt)$ (i.e. $i=k,1\le j \le k$) we consider the triangle
$T_{kj}=\Delta(M_{kj}M^T_{k0}M^T_{k-j,j})$ 
if $M_{ij}$  is an \emph{edge} node on the right edge (i.e. $j=k,1\le i \le k$) we 
chose $T_{ik}=\Delta(M_{ik}M^T_{0k}M^T_{i,k-i})$ and finally if  $M_{ij}$ is \emph{interior} (i.e. $1\le i,j \le k-1$) we define a triangle $T_{ij}=\Delta (M_{ij}M^T_{0j}M^T_{i0})$ (see Figure \ref{fig:variostri}). 
 \begin{figure}[h]
\resizebox{12cm}{5cm}{\includegraphics{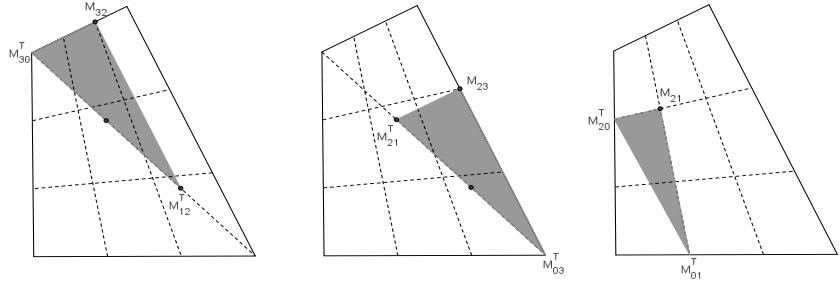}}
\caption{Representation of $T_{32}$, $T_{23}$ and $T_{21}$ in a $Q_3$ element {\it left:} $T_{32}$, {\it center:} $T_{23}$, {\it right:} $T_{21}$. }
 \label{fig:variostri}
\end{figure}
The geometry of these triangles are important in the sequel. In particular notice
that $T_{kj}$ and $T_{ik}$ are similar to the triangle $\Delta(V_2V_3V_4)$ therefore
we have immediately 
\begin{lemma}
\label{lema:similar}
Let $K(a,b,\at,\bt)$ be a general convex quadrilateral under $[\Delta 1, D2]$ (equiv. 
$[\Delta 2, D2]$, then for any $T=T_{kj}$ (resp. $T=T_{ik}$)  defined above we have that the side $\overline{M_{kj}M^T_{k0}}$ (resp. $\overline{M_{ik}M^T_{i\, k-i}}$)
is comparable to $l=\overline{V_3V_4}$ and the angle of $T$ at $M_{k0}$ (resp. $M^T_{i\,k-i}$)
is the angle $\alpha$ of condition $(D2)$. In particular $T$ verifies the $MAC$. 
\end{lemma}
For interior nodes we have the following 

\begin{lemma}
\label{lija}
Let $K=K(a,b,\at,\bt)$ be a convex quadrilateral which satisfy either $[D1,D2]$ or $[\Delta 1,D2,D3]$ (equiv. $[\Delta 2,D2,D3]$). 
If $1 \leq i,j \leq k-1$ then
\bit
\item[(a)]  $$|\overline{M_{i0}M_{ij}}|\comp a \ \ \mbox{and} \ \  |\overline{M_{0j}M_{ij}}|\comp b$$
(in particular $|\overline{M_{i0}M_{0j}}| \comp h$).
\item[(b)] $\alpha_{ij}$ is bounded away from $0$ and $\pi$ where $\alpha_{ij}$ is the angle between $\overline{M_{i0}M_{ij}}$ and $\overline{M_{i0}M_{0j}}$. In particular for $1\le i,j\le k-1$, any triangle  $T=\Delta(M_{i0},M_{0j},M_{ij})$ verifies the $MAC$.
\eit
\end{lemma}
\begin{proof}
 
\bit
\item[(a)] To prove that the measure of the segment $\overline{M_{i0}M_{ij}}$ is comparable to $a$ it is sufficient to prove that the measure of the segment $\overline{M_{i0}M_{ik}}$ is comparable to $a$ since these segments are mutually proportional.

For  $0<\hy=i/k<1$ have
$$\begin{array}{lcl}
|\overline{M_{i0}M_{ik}}|^2 &=&
\left\| F_K(1,\hy)-F_K(0,\hy) \right\|^2\ =\ \left\| (a(1-\hy)+\at\hy, (\bt-b)\hy) \right\|^2\\
\end{array}$$
therefore
$$a^2(1-\hy)^2 \le |\overline{M_{i0}M_{ik}}|^2 \le a^2(1-\hy)^2+2a\at\hy(1-\hy)+\hy^2|l|^2.$$
Using $(\Delta 1)$ (or $(D1)$) and that $l$ is comparable to the shortest side the statement is proved.
Similarly, to the appropriate $0<\hx<1$ we have
$$\begin{array}{lcl}
|\overline{M_{0j}M_{kj}}|^2 &=&
\left\| F_K(\hx,1)-F_K(\hx,0) \right\|^2\ =\ \left\| ((\at-a)\hx, b(1-\hx)+\bt\hx ) \right\|^2\\
\end{array}$$

therefore for a suitable constant $C$, we get
\bq
\label{uph}
b^2(1-\hx)^2  \le |\overline{M_{0j}M_{kj}}|^2 \leq 2\left[ (\at-a)^2\hx^2+ b^2 +(\bt-b)^2\hx^2 \right] \leq C(a^2+b^2)
\eq
and the proof concludes by using $(D3)$ in one case or Lemma \ref{lema:ladoacorto} in the other.
Now we immediately have, by using $\Delta 1$ or $D1$, that $|\overline{M_{i0}M_{0j}}| \comp h$.

\item[(b)] Calling $\mu$ the   
matrix with rows $w_1=M_{i0}-M_{ik}$ and $w_2=M_{i0}-M_{0j}$ we see that  
$$\frac{1}{\sin \alpha_{ij}}=\frac{\|w_1\|\|w_2\|}{|\det \mu|}.$$ 
Thanks to the previous item we know that the numerator can be bounded in terms of $ab$. We claim that $0<ab<C|\det \mu|$. Indeed, a direct
calculation gives for $y=\frac{i}{k}$ and $x=\frac{j}{k}$
$$|\det \mu|=aby \left| 1+\left( \frac{\at}{a}-1 \right)y+ \left(\frac{\bt}{b}-1 \right)x \right|$$
and since $1\le i\le k-1$ all we need to show is that term inside the modulus $m= 1+\left( \frac{\at}{a}-1 \right)y+ \left(\frac{\bt}{b}-1 \right)x$ stays away from zero. Now, if 
$\frac{\bt}{b}-1\ge 0$ then $m>1-y$ and we are done. On the other hand, if $\frac{\bt}{b}-1< 0$ then we write  
$$m \ge
\left( \frac{\bt}{b} +\frac{\at}{a}-1 \right)y+\frac{\bt}{b}(1-y)>\frac{\bt}{b}(1-y)$$
where the first inequality follows taking $x=1$ and the second one by \eqref{convJK}.
Using Lemma \ref{lema:cortoarriba} (see also Remark \ref{rem:alsodac}) the proof is complete.
$\Box$
\eit

\end{proof}
\section{The error treatment}
\label{section3}
\setcounter{equation}{0}

\begin{lemma}
\label{lema:desPoli} Let $K(a,b,\at,\bt)$ be a convex quadrilateral and assume $(\Delta 1)$. Let $T=T(a,b)$, then for any
polynomial $q\in \mathbb{P}_k$, there exists a constant
depending only on $k$ and on the $C$ given in $(\Delta 1)$ such that
\begin{equation}
\label{eq:desPoli}
\|q\|_{0,p,K(a,b,\at,\bt)}\le C \|q\|_{0,p,T} .
\end{equation}

\end{lemma}
\begin{proof} The proof is standard.  Let us
introduce an small rectangle $K_s$ and a large
rectangle $K_l$ as follows
$$K_s:=K(a/2,b/2,a/2,b/2)\subset T \subset K(a,b,\at,\bt)\subset K_l:=K(\overline{a},\overline{b},\overline{a},\overline{b})$$
where $\overline{a}=max\{a,\at\}, \overline{b}=max\{b,\bt\}$. All we need is to show that
\begin{equation}
\label{eq:normeq}
\|q\|_{0,p,K_l}\le C \|q\|_{0,p,K_s}. 
\end{equation}
Thanks to ($\Delta$1) we have that the quotients $\frac{\overline a}{a},\frac{\overline b}{b}$ are bounded in terms of a generic constant $C$. For the sake of clarity we rename this time  the constant and write $\bar C$. Consider now the \emph{reference} sets 
  $$\hat K_{\bar C}=K \left( \frac{1}{2\bar C},\frac{1}{2\bar C},\frac{1}{2\bar C},\frac{1}{2\bar C} \right)\subset \hat K= K(1,1,1,1).$$
Using equivalence of norms in the finite dimensional space $\mathbb{P}_k$ we get 
$$\|\hat q\|_{0,p,\hat K}\le C  \|\hat q\|_{0,p,\hat K_{\bar C}},$$
for any $\hat q\in \mathbb{P}_k$ and where $C$ depends only on $k$ and $\bar C$. 
Now \eqref{eq:normeq} follows by changing variables
with a linear map $L:\hat K \to K_l$ taking into account that for such an $L$,  $L(\hat K_{\bar C})\subset K_s$. \Box
\end{proof}
\smallskip

Write $K=K(a,b,\at,\bt)$ and let $\Pi_k$ be the Lagrange interpolation operator of order $k$ on the triangle $T=T(a,b)$ and let $p \geq 1$ then we can write

$$|u-Q_ku|_{1,p,K} \leq |u-\Pi_ku|_{1,p,K}+|\Pi_ku-Q_ku|_{1,p,K}.$$

Since $\Pi_ku - Q_ku$ belongs to the $\mathcal Q_k$ quadrilateral finite element space and vanishes at $M_{0j}$ and 
$M_{i0}$ for all $0 \leq i,j \leq k$, it follows that
$$\ds (\Pi_ku-Q_ku)(X) = \sum_{i,j \neq 0} (\Pi_ku-u)(M_{ij}) \phi_{ij}(X)$$

where $\phi_{ij}$ is the basis function associated to $M_{ij}$. Therefore
\bq
\label{destr}
|u-Q_ku|_{1,p,K} \leq |u-\Pi_ku|_{1,p,K} + \sum_{i,j \neq 0} |(\Pi_ku-u)(M_{ij})| |\phi_{ij}|_{1,p,K}.
\eq

Taking into account that $T$ verifies the $MAC$ (actually $MAC(\pi/2)$)   we have  \cite{A2,BA,J1} that
\bq
\label{eq:LagEnTri0}
\left\| u-\Pi_ku \right\|_{0,p,T} \leq C h^{k+1} |u|_{k+1,p,T},
\eq
and
\bq
\label{eq:LagEnTri}
|u-\Pi_ku|_{1,p,T} \leq C h^k |u|_{k+1,p,T}.
\eq
The next lemma extends this approximation result to $K$.

\begin{lemma}
\label{error1pik}
Let $K=K(a,b,\at,\bt)$ be a convex quadrilateral and assume $(\Delta 1)$. Let $T=T(a,b)$  and $\Pi_ku$   the $\mathbb{P}_k$ Lagrange interpolation operator on $T$. Then for any $1\le p$

\bq
\label{errorpikinK}
|u-\Pi_ku|_{1,p,K} \leq C h^k |u|_{k+1,p,K} .
\eq
\end{lemma}

\proof Let $u\in W^{k+1,p}(K)$ and $\mathcal{P}_k u \in \mathbb{P}_k$
defined as
$$\int_KD^\alpha u=\int_KD^\alpha \mathcal{P}_ku \qquad (|\alpha| \leq k).$$
Since $K$ is convex, we have
\bq
\label{uPku}
|u-\mathcal{P}_ku|_{1,p,K} \leq C h^k |u|_{k+1,p,K},
\eq
as one can see by applying repeatedly the Poincar\'e inequality. Writing
$$
|u-\Pi_ku|_{1,p,K} \le |u-\mathcal{P}_ku|_{1,p,K}
+|\mathcal{P}_ku-\Pi_ku|_{1,p,K},
$$
we observe that the first term is fine. For the second
 one we consider an arbitrary first derivative of $\mathcal{P}_ku-\Pi_ku$ and call it $ D(\mathcal{P}_ku-\Pi_ku)\in \mathbb{P}_{k-1}$. Using Lemma \ref{lema:desPoli} 
$$\|D(\mathcal{P}_ku-\Pi_ku)\|_{0,p,K}\le C \|D(\mathcal{P}_ku-\Pi_ku)\|_{0,p,T}\le C \left[ |\mathcal{P}_ku-u|_{1,p,K}+ |u-\Pi_ku|_{1,p,T} \right]$$
and the lemma follows from \eqref{eq:LagEnTri} and \eqref{uPku}. $\Box$

\begin{lemma}
\label{estphi}
Let $K=K(a,b,\at,\bt)$ be a convex quadrilateral.  
\bit
\item[(a)] Assume that $1\le p<3$ and that $K$ satisfies $[\Delta 1, D2]$ (equiv. $[\Delta 2, D2]$) then for any basis function  $\phi$, 
$$\ds |\phi|_{1,p,K} \leq C \frac{h^{1/p}}{|l|^{1/q}},$$
where $q$ is the conjugate exponent of $p$ (the constant $C$ may behave as $\frac{1}{3-p}$, see Remark \ref{rem:p<3}).

\item[(b)] Assume that $1\le p<3$ and that $K$ satisfies either $[\Delta 1, D2,D3]$ (equiv. 
$[\Delta 2, D2,D3]$) then 
$$\ds |\phi|_{1,p,K} \leq C \frac{h^{1/p}}{a^{1/q}},$$
where $\phi$ is an \emph{internal} basis function.

\item[(c)] For any $1\le p$, assume that $K$ satisfies $[D1,D2]$ then for any internal basis function
$$\ds |\phi|_{1,p,K} \leq C \frac{h^{1/p}}{a^{1/q}},$$

\item[(d)] For any $1\le p$, assume that $K$ satisfies $[D1,D2]$ then for any \emph{edge} basis function
$$\ds |\phi|_{1,p,K} \leq C \frac{h^{1/p}}{|l|^{1/q}}.$$
\eit

\end{lemma}

\begin{proof} Part (a) follows from Lemma \ref{edgesnodes} and Lemma \ref{lemaI}. On the other hand, by Lemma \ref{derparc} and Lemma \ref{lema:cortoarriba} we notice that to show (b) it is sufficient to prove that
\bq
\frac{b}{a^{p-1}} |1-\bt/b|^p I_p\ ,\ \frac{a}{b^{p-1}} (\at/a)^p I_p \leq
C \frac{h}{a^{p-1}}.
\eq
Using that $\at \leq |l|$ and $|b-\bt| \leq |l|$, together with $(\Delta 2)$ and Lemma \ref{lemaI} we have
$$\ds \frac{b}{a^{p-1}}|1-\bt/b|^p I_p \leq
C \frac{|l|^{p}}{b^{p}} \frac{b}{a^{p-1}} I_p \leq
C \frac{h}{b^{p-1}} \leq
C \frac{h}{a^{p-1}}$$
where the last inequality follows  from $(D3)$.
Similarly,
$$\ds \frac{a}{b^{p-1}} (\at/a)^p I_p \leq 
C \frac{h}{a^{p-1}}.$$
Item (c) follows similarly to item (b) using Lemma \ref{lemaIparadac} instead of Lemma \ref{lemaI} and
Lemma \ref{lema:ladoacorto} instead of $(D3)$. Finally, the last item (d) follows straightforwardly from Lemma \ref{edgesnodes} and  Lemma \ref{lemaIparadac}.  $\Box$
\end{proof}
Let us now recall the following
\begin{lemma}
\label{traza}
Let $T$ be a triangle with diameter $h_T$ and $e$ be any of its sides. For any $p \geq 1$ we have
$$\ds \left\| u \right\|_{0,p,e} \leq 2^{1/q} \left( \frac{|e|}{|T|} \right)^{1/p}
\left\{ \left\| u \right\|_{0,p,T}+h_T \left| u \right|_{1,p,T} \right\}$$
where $q$ is the dual exponent of $p$.
\end{lemma}

\begin{proof} See for instance \cite{V}. $\Box$
 
\end{proof}
Now we are ready to get bounds for 
$|(u-\Pi_k u)(M_{ij})|$. In order to do that
we consider the triangle $T_{ij}$ associated with
$M_{ij}$ defined in Section \ref{section:triang}.
\begin{lemma}
\label{upikumij}
Let $K=K(a,b,\at,\bt)$ be a convex quadrilateral satisfying either $[D1,D2]$ or $[\Delta 1,D2,D3]$ (equiv. $[\Delta 2,D2,D3]$).
For any $p \geq 1$ we consider its dual exponent  $q$. We have,
\bit
\item[(a)] (Edge nodes)  Assume either $i=k$ and $1\le j\le k$ or $j=k$ and $1\le i\le k$ then
$$|(u-\Pi_k u)(M_{ij})| \leq C \frac{|l|^{1/q}}{h^{1/p}} \left[ \left| u-\Pi_k u \right|_{1,p,T}+
h_{T} \left| u-\Pi_k u \right|_{2,p,T} \right],$$
where $T=T_{ij}$.
\item[(b)] (Interior nodes) If $1 \leq i,j \leq k-1$ then
$$|(u-\Pi_k u)(M_{ij})| \leq C \frac{a^{1/q}}{h^{1/p}} \left[ \left| u-\Pi_k u \right|_{1,p,T}+h_T \left| u-\Pi_k u \right|_{2,p,T} \right],$$
where $T=T_{ij}$.
\eit
\end{lemma}

\begin{proof}
\bit
\item[(a)] We write the case $i=k$ and $1 \leq j \leq k$ since the other one follows identically. Calling $e$ the side of $T=T_{kj}$ given by $e=\overline{M_{k0}M_{kj}}$ we get (by using H\"older's inequality, Lemma \ref{traza} and the fact that 
$(u-\Pi_k u)(M_{k0})=0$)
\bq
\label{trazaeT}
\begin{array}{lcl}
|(u-\Pi_k u)(M_{kj})| & \leq & \ds \int_{e} |\partial_e (u-\Pi_k u)|\ dx\\
\\
& \leq & |e|^{1/q} \left\| \partial_e (u-\Pi_k u) \right\|_{0,p,e}\\
\\
& \leq & 2^{1/q} \frac{|e|}{|T|^{1/p}} \left[ \left\| \partial_e (u-\Pi_k u) \right\|_{0,p,T}+
h_{T} \left| \partial_e (u-\Pi_k u) \right|_{1,p,T} \right] .
\end{array}
\eq
The item follows now by Lemma \ref{lema:similar}
that implies $\frac{|e|}{|T|^{1/p}}\le C\frac{|l|^{1/q}}{h^{1/p}}$.

\item[(b)] With the same ideas, consider $(u-\Pi_k u)(M_{i0})=0$ call $e=\overline{M_{i0}M_{ij}}$ and use now Lemma \ref{lija}. $\Box$
\end{itemize}
\end{proof}

\begin{lemma}
Let $K$ be a general convex quadrilateral
and $1 \leq i,j \leq k$,

\begin{enumerate}
 \item If $1\le p< 3$ and $K$ satisfies  $[\Delta 1,D2,D3]$ (equiv. $[\Delta 2,D2,D3]$), then \eqref{nodos1} holds.
 \item If $1\le p$ and $K$ satisfies 
 $[D1,D2]$, then \eqref{nodos1} holds.
\end{enumerate}
  
\bq
\label{nodos1}
|(u-\Pi_k u)(M_{ij})| |\phi_{ij}|_{1,p,K} \leq C h^k |u|_{k+1,p,K},
\eq
where $\phi_{ij}$ is the function basis associated to $M_{ij}$.
\end{lemma}

\proof The proof is essentially a combination of Lemmas \ref{estphi} and \ref{upikumij} together with the error estimation for triangles (\ref{eq:LagEnTri}), recalling that each $T_{ij}$ satisfies the maximum angle condition (Lemmas \ref{lema:similar} and \ref{lija}). $\Box$

\section{Main Theorem}
\label{section5}
\setcounter{equation}{0}

The  $L^p$ error estimate for a general convex quadrilateral was done in \cite{AM} for $k=1$ and any $p$. The argument used there works exactly in the same way for an arbitrary $k$. 

\begin{theorem}
\label{teoerrornormap}
Let $K$ be an arbitrary convex quadrilateral with diameter $h$. For any $1\le k$ and $1\le p$. There exists a constant $C$ independent of $K$ such that
\bq
\label{errornormap}
\left\| u-Q_k u \right\|_{0,p,K} \leq Ch^{k+1} |u|_{k+1,p,K}.
\eq
\end{theorem}
\begin{proof}
 See equation (41) in Theorem 6.1 of \cite{AM}, as well as  Lemma 6.1 in the same paper. $\Box$
\end{proof}
Now we can present our main result.
\begin{theorem} 
\label{theototal}
Let $K$ be a convex quadrilateral with diameter $h$ and $2\le k$ an integer:
\begin{enumerate}
 \item If $K$ satisfies $DAC(\psi_m,\psi_M)$, hence
 \eqref{eq:finalest} holds for any $1\le p$ with 
 $C=C(\psi_m,\psi_M)$.
 \item If $K$ satisfies $mac(\psi_m)$, hence
 \eqref{eq:finalest} holds for any $1\le p<3$ with 
 $C=C(\psi_m)$.
\end{enumerate}
\bq
\label{eq:finalest}
\left\| u-Q_k u \right\|_{0,p,K} + h \left| u-Q_ku \right|_{1,p,K} \leq Ch^{k+1} |u|_{k+1,p,K}.
\eq
\end{theorem}
 \begin{proof}
 Since the $L^p$ estimate holds for any convex quadrilateral, it is enough to prove 
\bq
\label{w1prdp}
\left| u-Q_ku \right|_{1,p,K} \leq Ch^k |u|_{k+1,p,K}.
\eq
Moreover, in order to prove (1) (resp. (2)) and thanks to Theorem \ref{teo:cardac} (resp. Theorem \ref{teo:carmac}) we can assume that  $K=K(a,b,\at,\bt)$ under $[D1,D2]$ (resp. either $[\Delta 1, D2, D3]$ or $[D1,D2]$). Therefore \eqref{w1prdp} follows from (\ref{destr}) combined with (\ref{errorpikinK}) and (\ref{nodos1}). $\Box$
 \end{proof}

 \smallskip
 To finish we present two counterexamples. In the first one we focus on the case  $1\le p<3$ showing a collection of elements
with uniform $RDP$ parameters (actually $RDP(\sqrt{5},3/4\pi)$) for which the constant in the $W^{1,p}$ interpolation error   blows up. The family does
not obey $mac$ although all the elements are under the $MAC(\frac{3\pi}{4}$). In particular,
the counterexample shows that the estimate may fail if an angle approaches cero. For the sake of simplicity we choose $k=2$.
\begin{cex}
\label{cexrdp} (Case $1\le p <3$)  For $0<s<1/2$ take $K=K(1,s,s,2s)$ 
\begin{figure}[h]
\resizebox{8cm}{6cm}{\includegraphics{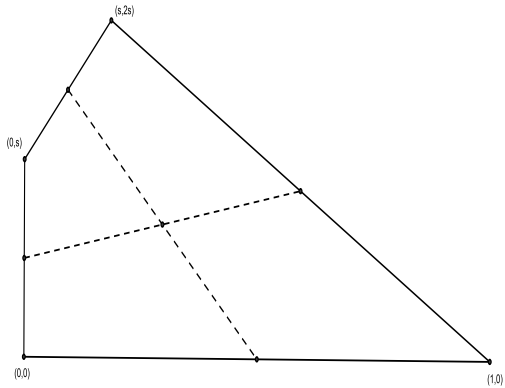}}
\caption{Representation of the quadrilateral $K(1,s,s,2s)$ and its nodes as a $\q_2$ element.}
 \label{fig:ex}
\end{figure}
and consider the function $u(x,y)=x(x-1/2)(x-1)$ which does not belong to the $\q_2$ space. Since $u(M_{0l})=0=u(M_{l0})$ for   $0\le l \le 2$ we have
$$Q_2 u = u(M_{11})\phi_{11}+u(M_{12})\phi_{12} + u(M_{22})\phi_{22} + u(M_{21})\phi_{21}$$ 
and therefore
$$\ds \frac{1}{s-1}\frac{\partial Q_2 u}{\partial y} =
\frac{(s-3)(s+1)}{2^6} \frac{\partial \phi_{11}}{\partial y}+
s \left( \frac{s+1}{2^3} \frac{\partial \phi_{12}}{\partial y} +
(s-1/2) \frac{\partial \phi_{22}}{\partial y} +
\frac{(s-2)}{2^3} \frac{\partial \phi_{21}}{\partial y} \right)$$
Then, for a suitable constant $C$ independent of $s$, we have
$$\left\| \frac{\partial \phi_{11}}{\partial y} \right\|_{0,p,K} \leq
C \left[ \left\| \frac{\partial Q_2 u}{\partial y} \right\|_{0,p,K} +
s \left( \left\| \frac{\partial \phi_{12}}{\partial y} \right\|_{0,p,K}+
\left\| \frac{\partial \phi_{22}}{\partial y} \right\|_{0,p,K} +
\left\| \frac{\partial \phi_{21}}{\partial y} \right\|_{0,p,K} \right) \right].$$

Using item (a) of Lemma \ref{estphi}  and taking into account that $h\comp 1$ and $|l|\comp s$ we have
$$\left\| \frac{\partial \phi_{11}}{\partial y} \right\|_{0,p,K} \leq
3C \left( \left\| \frac{\partial Q_2 u}{\partial y} \right\|_{0,p,K}+s^{1/p} \right).$$

Assume that \eqref{w1prdp} holds for this family. In that case we would have 
$$\left\| \frac{\partial Q_2 u}{\partial y} \right\|_{0,p,K}=
\left\| \frac{\partial (Q_2 u-u)}{\partial y} \right\|_{0,p,K} \leq
|Q_2u-u|_{1,p,K} \leq \bar{C} h^2 |u|_{3,p,K}$$

where $h^2 \comp 1$ and $|u|_{3,p,K} \comp |K|^{1/p} \comp s^{1/p}$. Consequently
\bq
\label{arriba}
\left\| \frac{\partial \phi_{11}}{\partial y} \right\|_{0,p,K} \leq C s^{1/p}.
\eq

On the other hand, a straightforward computation shows that
$$\ds \left( \frac{\partial \phi_{11}}{\partial y} \circ F_K \right)(\hx,\hy) =
\frac{2^4\hx[(s-1)\hy(\hx-\hy)+(1-\hx)(1-2\hy)]}{s[1+\hx+(s-1)\hy]}.$$

Therefore
$$\ds \left\| \frac{\partial \phi_{11}}{\partial y} \right\|_{0,p,K}^p =
\ds \int_{[0,1]^2} \frac{2^{4p}\hx^p | (s-1)\hy(\hx-\hy)+(1-\hx)(1-2\hy) |^p}{s^{p-1}[1+\hx+(s-1)\hy]^{p-1}}\ d\hx d\hy$$

Let $R=[0,1/8] \times [1/4,3/8] \subset [0,1]^2$. It is easy to check that on $R$ we have
\bq
(s-1)\hy(\hx-\hy)+(1-\hx)(1-2\hy) > (s-1)\hy(\hx-\hy) > 0  
\eq

which together with the fact $1+\hx+(s-1)\hy \leq 1+\hx$ allow us to obtain
\bq
\ds \left\| \frac{\partial \phi_{11}}{\partial y} \right\|_{0,p,K}^p \geq
\ds \frac{2^{4p}(1-s)^p}{s^{p-1}} \int_{R} \frac{\hx^p \hy^p(\hx-\hy)^p}{[1+\hx]^{p-1}}\ d\hx d\hy.
\eq

Since the function $\hy^p(\hx-\hy)^p$ is bounded below by a positive constant on $R$ and the function $\hx^p/(1+\hx)^{p-1}$ is integrable over this domain follows that
\bq
\label{abajo}
\ds \left\| \frac{\partial \phi_{11}}{\partial y} \right\|_{0,p,K} \geq C \frac{1}{s^{1/q}}
\eq

where $q$ is the dual exponent of $p$. Finally, combining (\ref{arriba}) with (\ref{abajo}) and taking $s \to 0$ we are lead to  a contradiction  and as a consequence the error estimate can not hold with a uniform constant $C$.  
\end{cex}

\begin{rem}
 Recall that  for $k=1$ and $1\le p<3$ the constant in the interpolation estimate can be bounded in terms of the constants given in the $RDP$ condition \cite{AM}. Actually, the interior node (available in $k=2$) plays a fundamental  role in the counterexample. This could lead the reader to the conclusion that removing internal nodes may help to weaken the conditions under which the estimate \eqref{errorinwp} holds. Regretfully this is not possible. Indeed from \cite{ABF} we know that  the accuracy of serendipity
 elements can be seriously deteriorated even for regular elements. The reason of that is the failure of the inclusion of $\mathbb{P}_k$ in the interpolation
 space. Our proof relies  strongly on this property  (see for instance the derivation of \eqref{destr}).

\end{rem}
\begin{figure}[h]
\resizebox{10cm}{5cm}{\includegraphics{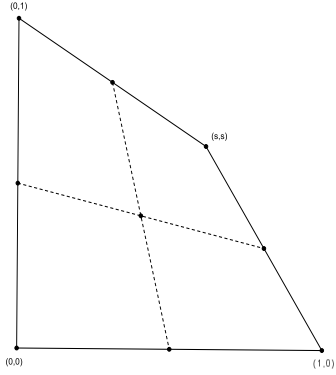}}
\caption{Representation of the quadrilateral $K(1,1,s,s)$ and its nodes as a $\q_2$ element.}
\label{fig:ex2}
\end{figure}
\begin{cex} (Case $3\le p$)
\label{cexrangep}
Consider the family $K(1,1,s,s)$, with $\frac{1}{2} < s \leq 5/8$, and the function $u(x,y)=x(x-1/4)(x-3/4)(x-3/8)(x-1)$. Observe that the maximum angle of $K=K(1,1,s,s)$ approaches $\pi$ as  $s \to \frac{1}{2}$ while $K$ verifies  $mac(\pi/4)$ for any value of $s$ in the selected range. Arguing as in previous counterexample we have
$$Q_2u = u(M_{11})\phi_{11}+u(M_{12})\phi_{12} + u(M_{21})\phi_{21} + u(M_{22})\phi_{22},$$
hence 
$$\ds \frac{\partial (Q_2u)}{\partial y} = u(M_{11}) \frac{\partial \phi_{11}}{\partial y}+
u(M_{12})\frac{\partial \phi_{12}}{\partial y} + 
u(M_{21})\frac{\partial \phi_{21}}{\partial y} + 
u(M_{22})\frac{\partial \phi_{22}}{\partial y}.$$

Observe that $u(M_{11}), u(M_{12})$ and $u(M_{21})$ are polynomial expressions in the variable $s$ having $1/2$ as a single root. Therefore we  can write
$$u(M_{ij})=(s-1/2) q_{ij}(s)$$
for each $(i,j) \in I=\{ (1,1), (1,2), (2,1) \}$ where $q_{ij}$ is a polynomial. On the other hand $|u(M_{22})|>C>0$ if $\frac{1}{2} < s \leq 5/8$. Therefore

\bq
\label{phi22}
\ds \left\| \frac{\partial \phi_{22}}{\partial y} \right\|_{0,p,K} \leq 
C \left[ \left\| \frac{\partial Q_2u}{\partial y} \right\|_{0,p,K}   
+ (s-1/2) \sum_{(i,j) \in I} \left\| \frac{\partial \phi_{ij}}{\partial y} \right\|_{0,p,K}
\right].
\eq
If the error estimates holds then
$$\left\| \frac{\partial Q_2u}{\partial y} \right\|_{0,p,K} =
\left\| \frac{\partial Q_2u-u}{\partial y} \right\|_{0,p,K} \leq 
| Q_2u -u|_{1,p,K} \leq C |u|_{3,p,K},$$
since $h\comp 1$ and as a consequence
\bq
\label{Q2u}
\left\| \frac{\partial Q_2u}{\partial y} \right\|_{0,p,K} \leq C.
\eq
On the other hand for $1/2 < s \leq 5/8$ we readily notice that $\sin(\alpha) \comp (s-1/2)$. 
 Then combining this with \eqref{eq:IPintermsofsin} 
and Lemma \ref{edgesnodes} we get 
\bq
\label{sumphiij}
\ds \sum_{(i,j) \in I} \left\| \frac{\partial \phi_{ij}}{\partial y} \right\|_{0,p,K} \leq C \frac{1}{(s-1/2)^{1/q}}
\eq
where $q$ is the dual exponent to $p$. Finally, (\ref{sumphiij}) combined with (\ref{phi22}) with (\ref{Q2u}) give us
\bq
\label{phi22above}
\ds \left\| \frac{\partial \phi_{22}}{\partial y} \right\|_{0,p,K} \leq C
\eq
for some positive constant. However, a straightforward calculation yields
$$\ds \left( \frac{\partial \phi_{22}}{\partial y} \circ F_K \right) (\hx,\hy) = 
\frac{2 \hx \left[ (s-1)\hy(\hx-\hy)+(2\hx-1)(4\hy-1) \right]}{1+(s-1)(\hx+\hy)}$$
hence
$$\ds \left\| \frac{\partial \phi_{22}}{\partial y} \right\|_{0,p,K}^p = 
\int_0^1 \int_0^1 
\frac{(2 \hx \left[ (s-1)\hy(\hx-\hy)+(2\hx-1)(4\hy-1) \right])^p}{(1+(s-1)(\hx+\hy))^{p-1}}\ d\hx d\hy.$$
Let $T$ be the triangle with vertices $(3/4,3/4), (3/4,1)$ and $(1,1)$. It is easy to check that 
$$\ds \left\| \frac{\partial \phi_{22}}{\partial y} \right\|_{0,p,K}^p \geq 
\ds C \int_T \frac{1}{(1+(s-1)(\hx+\hy))^{p-1}}\ d\hx d\hy,$$
and integrating explicitly for $p>3$ we get
$$\ds \left\| \frac{\partial \phi_{22}}{\partial y} \right\|_{0,p,K}^p \geq 
C \frac{(2s-1)^{3-p}/2+(3s-1)^{3-p}/2^{4-p}-(7s-3)^{3-p}/4^{3-p}}{(s-1)^2(2-p)(3-p)}$$
and hence $\ds \left\| \frac{\partial \phi_{22}}{\partial y} \right\|_{0,p,K}^p \to \infty$ if $s \to 1/2$. Since this fact contradicts (\ref{phi22above}) we conclude that the error estimate does not hold. The case $p=3$ follows similarly.  
\end{cex}

\end{document}